\documentclass[12pt, reqno]{amsart}
\usepackage{amsmath, amsthm, amscd, amsfonts, amssymb, graphicx, color}
\usepackage[bookmarksnumbered, colorlinks, plainpages]{hyperref}
\hypersetup{colorlinks=true,linkcolor=red, anchorcolor=green, citecolor=cyan, urlcolor=red, filecolor=magenta, pdftoolbar=true}

\newtheorem{theorem}{Theorem}[section]
\newtheorem{lemma}[theorem]{Lemma}
\newtheorem{proposition}[theorem]{Proposition}
\newtheorem{corollary}[theorem]{Corollary}
\theoremstyle{definition}

\theoremstyle{remark}
\newtheorem{remark}[theorem]{Remark}
\numberwithin{equation}{section}

\begin{document}

\title[A variance bound for a general function]{A variance bound for a general function of independent noncommutative random variables}

\author[A. Talebi and M.S. Moslehian] {Ali Talebi and Mohammad Sal Moslehian}

\address{ Department of Pure Mathematics, Center of Excellence in Analysis on Algebraic Structures (CEAAS), Ferdowsi University of
Mashhad, P.O. Box 1159, Mashhad 91775, Iran.}
\email{alitalebimath@yahoo.com}
\email{moslehian@um.ac.ir and moslehian@member.ams.org}

\subjclass[2010]{Primary 46L53; Secondary 60E15.}

\keywords{Efron--Stein inequality; Random matrix; Noncommutative probability; Trace; Conditional expectation.}

\begin{abstract}
The main purpose of this paper is to establish a noncommutative analogue of the Efron--Stein inequality, which bounds the variance of a general function of some independent random variables. Moreover, we state an operator version including random matrices, which extends a result of D. Paulin et al. [Ann. Probab. 44 (2016), no. 5, 3431--3473]. Further, we state a Steele type inequality in the framework of noncommutative probability spaces.
\end{abstract} \maketitle
%
%
\section{Introduction and Preliminaries}

Assume that $Z = f\left(X_1, X_2, \ldots, X_n \right)$ is a symmetric function of independent random variables $X_1, X_2, \ldots, X_n$ in a probability space $\left( \Omega, \mathcal{F}, \mathbb{P}\right)$. Efron and Stein investigated the Tukey jackknife estimate of variance for the statistic $Z$; cf. \cite{ES, STE}. They proved that
\begin{equation*}
{\rm var}(Z) \leq \frac{1}{2} \sum_{j=1}^n \mathbb{E}\left[\left(Z - Z_j^\prime \right)^2\right],
\end{equation*}
where $X_1^\prime, X_2^\prime, \ldots, X_n^\prime$ are independent copies of $X_1, X_2, \ldots, X_n$, respectively, and\\
$Z_j^\prime = f\left(X_1, \ldots, X_j^\prime, \ldots, X_n\right)$.

Next, Steele \cite{STE} established a new version of the Efron--Stein inequality. He proved that with the above notation, if $Z_j = f_j \left(X_1, \ldots, X_{j-1}, X_{j+1}, \ldots, X_n\right)$, where $f_j$ is a measurable function of $n-1$ variables, then
\begin{equation*}
{\rm var}(Z) \leq \frac{1}{2} \sum_{j=1}^n \mathbb{E}\left[\left(Z - Z_j \right)^2\right].
\end{equation*}
In this paper, we present a noncommutative Efron--Stein inequality, which improves the classical Efron--Stein inequality for random variables. Further, we state a Steele type inequality in the setting of noncommutative probability spaces.
Recently, Paulin et al. \cite{Pa} established some Efron--Stein inequalities describing the concentration properties of a matrix-valued function of independent random variables. In addition, we establish a matrix version of our noncommutative Efron--Stein inequality.
Throughout the paper, $(\mathcal{M}, \tau)$ denotes a noncommutative probability space, that is, a von Neumann algebra $\mathcal{M}$ equipped with a normal faithful finite trace $\tau$ with $\tau(1)=1$, where $1$ stands for the identity of $\mathcal{M}$. We denote the self-adjoint elements of $\mathcal{M}$ by $\mathcal{M}_{sa}$.
Let $1 \leq p < \infty$. The Banach space $L_p(\mathcal{M})$ is the completion of $\mathcal{M}$ with respect to the $p$-norm $\|x\|_p:=\tau(|x|^p)^{1/p}$. The elements of $L_1\left(\mathcal{M} \right)$ are called (noncommutative) random variables.\\
Let $x$ be a normal random variable and $e^x$ be the unique spectral measure on the Borel subsets $\mathfrak{B}(\mathbb{C})$ of the complex plane $\mathbb{C}$.\\
Assume that $\mathcal{N}$ is a von Neumann subalgebra of $\mathcal{M}$. Then there exists a map $\mathcal{E}_{\mathcal{N}}: \mathcal{M} \longrightarrow \mathcal{N}$, satisfying the following properties:
\begin{enumerate}
\item[(i)] $\mathcal{E}_{\mathcal{N}}$ is normal positive contractive projection from $\mathcal{M}$ onto $\mathcal{N}$;
\item[(ii)] $\mathcal{E}_{\mathcal{N}}(axb) = a \mathcal{E}_{\mathcal{N}}(x) b$ for every $x \in \mathcal{M}$ and $a, b \in \mathcal{N}$;
\item[(iii)] $\tau \circ \mathcal{E}_{\mathcal{N}} = \tau$.
\end{enumerate}
Moreover, $\mathcal{E}_{\mathcal{N}}$ is the unique map verifying {\rm (ii)} and {\rm (iii)}.
It is known that $\mathcal{E}_{\mathcal{N}}$ can be extended to a contractive positive projection, denoted by the same $\mathcal{E}_{\mathcal{N}}$, from $L_p(\mathcal{M})$ onto $L_p(\mathcal{N})$, named the conditional expectation with respect to $\mathcal{N}$.
The reader is referred to \cite{Bek, sm1, sm2} and references therein for more information on noncommutative probability spaces.

Let $d \in \mathbb{N}$ and $\mathbb{M}_d (\mathbb{C}) \otimes \mathcal{M} \cong \mathbb{M}_d ( \mathcal{M} )$ be the algebra of all $d \times d$ matrices with entires in $\mathcal{M}$. Then the trace $\overline{\tau}$ is defined by
\begin{eqnarray*}
\overline{\tau} \left( X \right) = \tau\left( \overline{{\rm tr}}(X) \right) := \frac{1}{d} \sum_{i=1}^d \tau (X_{ii}).
\end{eqnarray*}
It is a normalized normal faithful finite trace (see \cite[Proposition 2.14]{Tak}) on $\mathbb{M}_d ( \mathcal{M} )$, where $X = \left(X_{ij} \right)_{d \times d}$ and $\overline{{\rm tr}} = \frac{1}{d} {\rm tr}$.\\
Similarly, let $\mathbb{M}_d ( \mathcal{N} )$ denote the subalgebra of $d \times d$ matrices with entires in $\mathcal{N}$. Then $\mathbb{M}_d ( \mathcal{N} )$ is a von Neumann subalgebra of $\mathbb{M}_d ( \mathcal{M} )$ and the corresponding conditional expectation $\mathbb{E} = I_{\mathbb{M}_d} \otimes \mathcal{E}_{\mathcal{N}}$ is given by
$
\mathbb{E}\left( X \right) = \left( \mathcal{E}_{\mathcal{N}} \left( X_{ij} \right) \right)_{d \times d}.
$
Note that $\tau (X)$ is defined by $\left( \tau \left(X_{ij} \right) \right)_{d \times d}$.

\section{Noncommutative Efron--Stein inequality}

Throughout this section, for von Neumann subalgebras $\mathcal{N}_1, \mathcal{N}_2, \ldots, \mathcal{N}_n$ of $\mathcal{M}$ let $\mathcal{E}_j$ and $\mathcal{E}_j^{\prime}$ denote the conditional expectations on $\mathcal{M}$ with respect to $W^*\left(\bigcup_{k \neq j} \mathcal{N}_k\right)$ and $W^*\left(\bigcup_{k=1}^j \mathcal{N}_k\right)$, respectively, where by $W^*(S)$ we mean the von Neumann algebra generated by $S$ for a subset $S\subset \mathcal{M}$. We use the notation $W^*(x)$ for the von Neumann subalgebra $W^*\left(\{e^x(B) : B \in \mathfrak{B}(\mathbb{C}) \}\right)$ of $\mathcal{M}$.\\

The variance of an element $x \in \mathcal{M}$ is defined by ${\rm var}(x) := \tau\left( (x-\tau(x))^2 \right)$. The following formula seems to be known in the literature:
\begin{equation*}
{\rm var}(x) = \inf_{\lambda \in \mathbb{R}} \tau\left( (x - \lambda 1)^2 \right).
\end{equation*}
Recall that independence, is one of the main notions in the classical probability. In the noncommutative probability setting, the concept of independence also plays an essential role. There are more noncommutative notions of independence; cf. \cite{VDN} and see also \cite{Jia}. The condition $\tau(xy)=\tau(x)\tau(y)$ in the next proposition is known as Boolean independence \cite{L}. More precisely, two subalgebras $\mathcal{N}_1$ and $\mathcal{N}_2$ are said to be Boolean independent if $\tau(ab)=\tau(a)\tau(b)$ for any $a \in \mathcal{N}_1$ and $b \in \mathcal{N}_2$ and two random variables $x, y$ are Boolean independent if $W^*(x)$ and $W^*(y)$ are Boolean independent.
\begin{proposition}
Let $x \in \mathcal{M}_{sa}$. Then
\begin{equation}\label{mos}
{\rm var}(x) =\inf\{\tau\left((x-y)^2\right): ~y\in \mathcal{M}_{sa} {\rm ~and~} x, y {\rm ~are~ Boolean~independent}\}.
\end{equation}
\end{proposition}
\begin{proof}
Let $y\in \mathcal{M}_{sa}$ and $\tau(xy)=\tau(x)\tau(y)$. First note that
\begin{equation}\label{eq55}
\tau (y)^2 = \tau(y.1)^2 \leq \tau\left(y^2\right) \tau (1) = \tau \left(y^2\right).
\end{equation}
Next we have
\begin{eqnarray*}
{\rm var}(x) &=& \tau \left(x^2\right) - \tau (x)^2\\
&\leq & \tau \left(x^2\right) - \tau (x)^2 + \left(\tau (x) - \tau (y)\right)^2\\
&=& \tau\left(x^2\right) - 2\tau(x) \tau(y) + \tau(y)^2\\
&\leq & \tau\left(x^2\right) - 2\tau (xy) + \tau (y^2) \\
&=& \tau\left( (x-y)^2\right).
\end{eqnarray*}
It follows from the definition of variance that \eqref{mos} is valid.
\end{proof}
The expression $\tau\left((x-y)^2\right)$, in particular when $y$ is replaced by $\mathcal{E}(x)$, is in the spirit of Efron--Stein inequality.\\
In this paper, we need also the tensor and free independence to get our results; cf \cite{D}. Von Neumann subalgebras $\mathcal{N}_j\,\,(1\leq j\leq n)$ are called tensor independent if
\begin{eqnarray*}
\tau\left(\prod_{i=1}^m\left(\prod_{k=1}^n a_{ki}\right) \right) = \prod_{k=1}^n\tau\left(\prod_{i=1}^m a_{ki}\right),
\end{eqnarray*}
whenever $a_{kj} \in \mathcal{N}_k$ ~ ($1 \leq j \leq m$; $1 \leq k \leq n$; $m \in \mathbb{N}$);\\
and the $\mathcal{N}_j\,\,(1\leq j\leq n)$ are free independent if
\begin{align*}
\tau\left(a_1a_2\ldots a_m \right)=0,
\end{align*}
when $a_k \in \mathcal{N}_{i_k}$, $i_1 \neq i_2 \neq \ldots \neq i_m$ and $\tau(a_k)=0$ for all $k \in \{1, 2, \ldots, m\}$.

To achieve our main result of this section, we model some techniques of the classical theory \cite{Bouch} to the context of noncommutative probability spaces.
We need the next lemma to get the main result. To prove the lemma, we apply the following theorem, which relates to the notion of conditional independence.
\begin{theorem}\cite[Theorem 5.1]{D}\label{th4}
Let $x \in L_1\left(\mathcal{M}\right)$ and $\mathcal{M}_1, \mathcal{M}_2$ be von Neumann subalgebras of $\mathcal{M}$ such that $W^*\left(\{x\}\cup \mathcal{M}_1\right)$ and $\mathcal{M}_2$ are either freely independent or tensor independent. Then $\mathcal{E}_{\mathcal{M}_1, \mathcal{M}_2} (x) = \mathcal{E}_{\mathcal{M}_1}(x)$, where $\mathcal{E}_{\mathcal{M}_1, \mathcal{M}_2}$ is the conditional expectation with respect to $W^*\left(\mathcal{M}_1 \bigcup \mathcal{M}_2 \right)$.
\end{theorem}

\begin{lemma}\label{le1}
Let $x_j \in L_1(\mathcal{M})\,\, (1\leq j\leq n)$ be either freely independent or tensor independent self-adjoint random variables and $f$ be a function from $L_1(\mathcal{M})_{sa} \times \ldots \times L_1(\mathcal{M})_{sa}$ into $L_1(\mathcal{M})_{sa}$. Then
\begin{equation*}
{\rm var}\left(f(x_1, x_2, \ldots, x_n)\right) \leq \sum_{i=1}^n \tau\left((f(x_1, x_2, \ldots, x_n) - \mathcal{E}_i(f(x_1, x_2, \ldots, x_n)))^2\right).
\end{equation*}
\end{lemma}
\begin{proof}
Note that if $\mathcal{N}$ is an arbitrary von Neumann subalgebra of $\mathcal{M}$, $y \in L_1(\mathcal{N})$ and $\mathcal{E}_{\mathcal{N}}$ is the conditional expectation with respect to $\mathcal{N}$, then from the property (iii) of the definition of conditional expectation we obtain
\begin{eqnarray}\label{eqn1}
\tau(xy) = \tau \left(\mathcal{E}_{\mathcal{N}}(xy)\right)
= \tau (y \mathcal{E}_{\mathcal{N}}(x)),
\end{eqnarray}
for all $x$ in $L_1(\mathcal{M})$.\\
Consider the von Neumann subalgebras $\mathcal{N}_{j} = W^*(x_j)$ for any $1 \leq j \leq n$. Put $y:= f(x_1, x_2, \ldots, x_n)$ and $z:= y - \tau(y)$. Set $z_j =\mathcal{E}_j^\prime (y ) - \mathcal{E}_{j-1}^\prime (y)$, which is clearly a self-adjoint element. We have $z= \sum_{j=1}^n z_j$, since
\begin{eqnarray*}
\sum_{j=1}^n z_j = \sum_{j=1}^n \left(\mathcal{E}_j^\prime (y ) - \mathcal{E}_{j-1}^\prime (y)\right)
= \mathcal{E}_n^\prime (y ) - \mathcal{E}_{0}^\prime (y)
= y - \tau(y) 1
= z,
\end{eqnarray*}
where $\mathcal{E}_0 ^\prime$ is the conditional expectation with respect to $\mathbb{C}$.\\
We have
\begin{eqnarray}\label{eq1}
{\rm var}(y) \nonumber &=& \tau\left(z^2\right)\\ \nonumber
&=& \tau \left(\left(\sum_{j=1}^n z_j \right)^2\right)\\ \nonumber
&=& \tau \left(\sum_{j=1}^n z_j^2\right) + \tau \left(\sum_{i < j} z_i z_j\right) + \tau \left(\sum_{j < i} z_i z_j\right)\\
&=& \sum_{j=1}^n \tau\left(z_j ^2\right).
\end{eqnarray}
Notice that $\tau(z_iz_j)=0$, for all $i \neq j$, since if, e.g. $i > j$
\begin{eqnarray*}
\tau (z_i z_j) &=& \tau \left(\mathcal{E}_j^\prime(z_i z_j)\right)\\
&=& \tau \left(\mathcal{E}_j^\prime(z_i)z_j \right) \qquad (\text{by \ref{eqn1}})\\
&=& \tau \left( \mathcal{E}_j^\prime(\mathcal{E}_i^\prime (y ) - \mathcal{E}_{i-1}^\prime (y)) z_j \right)\\
&=& \tau \left( (\mathcal{E}_j^\prime (y) - \mathcal{E}_j^\prime (y))z_j \right)\\
&=& 0.
\end{eqnarray*}
We claim that
\begin{equation}\label{eq2}
\tau\left(z_j ^2\right) \leq \tau\left((y - \mathcal{E}_j(y))^2\right),
\end{equation}
for every $1 \leq j \leq n$.\\
It follows from independence property that
\begin{equation}\label{eqn2}
\mathcal{E}_j^\prime \circ \mathcal{E}_j = \mathcal{E}_{j-1}^\prime= \mathcal{E}_j^\prime \circ \mathcal{E}_{j-1}^\prime;~ j=1, \ldots, n.
\end{equation}
To achieve it, due to $\mathcal{E}_{j}^\prime \circ \mathcal{E}_{j-1}^\prime = \mathcal{E}_{j-1}^\prime$, it is enough to show that
\begin{align}\label{eqn4}
\mathcal{E}_j^\prime (x) = \mathcal{E}_{j-1}^\prime(x),
\end{align}
for any $x \in L_1\left(W^*\left(\bigcup_{k \neq j} \mathcal{N}_k\right)\right)$.\\
Fix $x \in W^*\left(\bigcup_{k \neq j} \mathcal{N}_k\right)$. In the case when $x_j \,\, (1\leq j\leq n)$ are freely independent, it follows from \cite[Proposition 2.5.5]{VDN} that $W^*\left(\bigcup_{k \neq j} \mathcal{N}_k\right)$ and $\mathcal{N}_j$ are freely independent. Therefore equation \eqref{eqn4} can be deduced from Theorem \ref{th4}, by considering $\mathcal{M}_1 = \bigcup_{k=1}^{j-1} \mathcal{N}_k$ and $\mathcal{M}_2 = \mathcal{N}_j$ if $j > 1$. If $j=1$, then, obviously $\mathcal{E}_1^\prime (x) = \tau(x)$ for all $x \in L_1\left(W^*\left(\bigcup_{k \neq 1} \mathcal{N}_k\right)\right)$. In fact, by the $\|.\|_1$-continuity of $\mathcal{E}_1^\prime$ and that of $\tau$ as well as a density argument, we may suppose that $x \in W^*\left(\bigcup_{k \neq 1} \mathcal{N}_k\right)$. Since both of sides are elements of $\mathcal{N}_1$, it suffices to prove that $\tau\left(\mathcal{E}_1^\prime (x)y\right) = \tau\left(\tau(x)y\right)$ for any $y \in \mathcal{N}_1$. This indeed holds by virtue of $\tau\left(\mathcal{E}_1^\prime (x)y\right) = \tau\left(\mathcal{E}_1^\prime (xy)\right) = \tau\left(xy\right) = \tau(x)\tau(y)$. We remark that the last equality follows from expanding of the equation $\tau\left( \left(x- \tau(x)\right)\left(y - \tau(y)\right)\right) = 0$, which can be deduced from the free independence. \\
Now, assume that $x_j\,\, (1\leq j\leq n)$ are tensor independent. Then $W^*\left(\bigcup_{k \neq j} \mathcal{N}_k\right)$ and $\mathcal{N}_j$ are also tensor independent; i.e.
\begin{eqnarray*}
\tau\left(a_{11}a_{j1}a_{12}a_{j2}\ldots a_{1m}a_{jm}\right) =\tau\left(\prod_{i=1}^m a_{1i}\right)\tau\left(\prod_{i=1}^m a_{ji}\right),
\end{eqnarray*}
whenever $a_{1i} \in W^*\left(\bigcup_{k \neq j} \mathcal{N}_k\right), a_{ji} \in \mathcal{N}_j$ ~ ($1 \leq i \leq m$; $m \in \mathbb{N}$). Indeed, from the ultra-weak continuity of $a \mapsto \tau(ab)$ and the ultra-density of the set of all polynomials $g$ in elements of $\bigcup_{k \neq j} \mathcal{N}_k$ in $W^*\left(\bigcup_{k \neq j} \mathcal{N}_k\right)$ (cf. \cite{Tak}), it is enough that the equation
\begin{eqnarray*}
\tau\left(g_{1}a_{j1}g_{2}a_{j2}\ldots g_{m}a_{jm}\right) =\tau\left(\prod_{i=1}^m g_{i}\right)\tau\left(\prod_{i=1}^m a_{ji}\right),
\end{eqnarray*}
is valid for such polynomials. Moreover, by linearity, we may assume that $g_i = u_{1}^{(i)}u_{2}^{(i)}\ldots u_{j-1}^{(i)}u_{j+1}^{(i)}\ldots u_{n}^{(i)}$ with $u_{k}^{(i)} \in \mathcal{N}_k; ~ k \in \{1, 2, \ldots, j-1, j+1, \ldots, n\}$, and this can easily be shown as follows (note that the $u_{k}^{(i)}$ are allowed to be equal to $1$ for $1 \in \mathcal{N}_k$).\\
\begin{align*}
& \tau\left(g_{1}a_{j1}g_{2}a_{j2}\ldots g_{m}a_{jm}\right)\\
&=\tau\Big(\left(u_{1}^{(1)}u_{2}^{(1)}\ldots u_{j-1}^{(1)} 1 u_{j+1}^{(1)}\ldots u_{n}^{(1)}\right)\left(1 \ldots 1 a_{j1} 1\ldots 1\right) \\
& \quad \ldots \left(u_{1}^{(m)}u_{2}^{(m)}\ldots u_{j-1}^{(m)} 1 u_{j+1}^{(m)}\ldots u_{n}^{(m)}\right)\left(1 \ldots 1 a_{jm} 1\ldots 1\right)\Big)\\
&=  \tau\left(u_{1}^{(1)}u_{1}^{(2)}\ldots u_{1}^{(m)}\right) \ldots \tau\left(u_{j-1}^{(1)}u_{j-1}^{(2)}\ldots u_{j-1}^{(m)}\right)\\
& \quad \cdot \tau\left(a_{j1}a_{j2}\ldots a_{jm}\right) \tau\left(u_{j+1}^{(1)}u_{j+1}^{(2)}\ldots u_{j+1}^{(m)}\right) \ldots  \tau\left(u_{n}^{(1)}u_{n}^{(2)}\ldots u_{n}^{(m)}\right)\\
&= \tau\left(a_{j1}a_{j2}\ldots a_{jm}\right) \tau\Big(\left(u_{1}^{(1)}u_{2}^{(1)}\ldots u_{j-1}^{(1)} 1 u_{j+1}^{(1)}\ldots u_{n}^{(1)}\right) \\
& \quad \ldots \left(u_{1}^{(m)}u_{2}^{(m)}\ldots u_{j-1}^{(m)} 1 u_{j+1}^{(m)}\ldots u_{n}^{(m)}\right)\Big)\qquad\qquad ({\rm ~by~ tensor~ independence}) \\
&= \tau\left(a_{j1}a_{j2}\ldots a_{jm}\right) \tau\left(g_{1}g_{2}\ldots g_{m}\right).
\end{align*}
Thus equation \eqref{eqn4} is true for every $x \in W^*\left(\bigcup_{k \neq j} \mathcal{N}_k\right)$ and hence, by a density argument and the $\|.\|_1$-contractivity of the conditional expectations, for each $x \in L_1\left(W^*\left(\bigcup_{k \neq j} \mathcal{N}_k\right)\right)$. Again, by putting $\mathcal{M}_1 = \bigcup_{k=1}^{j-1} \mathcal{N}_k$ and $\mathcal{M}_2 = \mathcal{N}_j$, if $j > 1$, equation \eqref{eqn4} can be obtained from Theorem \ref{th4}. In the case when $j=1$, the investigation of $\mathcal{E}_1^\prime (x) = \tau(x)$ is similar to the freeness case. The second equation in \eqref{eqn2} is evident.\\
Therefore
\begin{eqnarray}\label{eqn3}
z_j^2 &=& \left(\mathcal{E}_j^\prime (y ) - \mathcal{E}_{j-1}^\prime (y)\right)^2
= \left(\mathcal{E}_j^\prime (y - \mathcal{E}_j(y)\right)^2.
\end{eqnarray}
Thus
\begin{eqnarray*}\label{eq22}
\tau \left(z_j^2\right) &=& \left\| z_j^2 \right\|_1 \\
&=& \left\| \left( \mathcal{E}_j^\prime \left(y - \mathcal{E}_j(y)\right)\right)^2 \right\|_1 \qquad\qquad\qquad \qquad\qquad\qquad\qquad\qquad\quad (\text{by \eqref{eqn3}})\\
&=& \left\| \mathcal{E}_j^\prime \left(y - \mathcal{E}_j(y)\right) \right\|_2^2 \\
&\leq & \left\| y - \mathcal{E}_j(y) \right\|_2^2 \quad(\text{by extending of $\mathcal{E}_j^\prime$, which is still a contraction.})\\
&=& \tau \left((y - \mathcal{E}_j(y))^2 \right),
\end{eqnarray*}
which gives us inequality \eqref{eq2}.\\
The desired inequality can be now deduced from (\ref{eq1}) and (\ref{eq2}).
\end{proof}

It is known that if $X$ and $Y$ are independent identically distributed, then $\mathbb{E}| XY | = \mathbb{E}| X | \mathbb{E} | Y |$ and $\mathbb{E} \left( |X|^p \right) = \mathbb{E} \left(|Y|^p \right)$ for any $p \geq 1$; cf. \cite{G}. Similarly, if $\mathcal{G}$ is a sub-$\sigma$-algebra of $\mathcal{F}$ and $X$ and $Y$ are independent identically distributed conditionally to $\mathcal{G}$, then $\mathbb{E}_{\mathcal{G}}| XY | = \mathbb{E}_{\mathcal{G}}| X | \mathbb{E}_{\mathcal{G}} | Y |$ and $\mathbb{E}_{\mathcal{G}} \left( |X|^p \right) = \mathbb{E}_{\mathcal{G}} \left(|Y|^p \right)$ for any $p \geq 1$, where $\mathbb{E}_{\mathcal{G}}$ is the conditional expectation relative to $\mathcal{G}$.\\
 In the noncommutative setup, we use the notation $\mathbb{P}\left( x \geq t \right) := \tau \left(\chi_{[t, \infty)}(x) \right)$. In this case, if $x, y \in \mathcal{M}_+$ are Boolean independent identically distributed random variables, then $\tau (xy) = \tau(x)\tau(y)$ and as mentioned in \cite{sm1, TMS} we deduce that
 \begin{eqnarray*}
 \| x \|_p^p = \int_0^\infty p t^{p-1} \tau \left(\chi_{[t, \infty)}(x)\right) dt = \int_0^\infty p t^{p-1} \tau \left(\chi_{[t, \infty)}(y)\right) dt = \| y \|_p^p.
 \end{eqnarray*}
The conditionally independent and identically distributed random variables can be defined in the same way.\\
Now we are ready to present our generalization of the Efron--Stein inequality.
\begin{theorem}[Noncommutative Efron--Stein inequality]\label{Efr1}
Let $x_1, \cdots, x_n   $ be either freely independent or tensor independent self-adjoint random variables, and $f: L_1(\mathcal{M})_{sa} \times \ldots \times L_1(\mathcal{M})_{sa} \longrightarrow L_1(\mathcal{M})_{sa}$ be an arbitrary function. Suppose that $y= f\left(x_1, x_2, \ldots, x_n\right)$ and $y_j^\prime = f\left(x_1, \ldots, x_{j-1},x_j^\prime, x_{j+1}, \ldots, x_n\right)$ are such that $\tau(yy_j^\prime) \leq \| \mathcal{E}_j(y)\|_2^2$ and $\tau \left((y_j^\prime)^2\right) \geq \tau \left( y^2 \right)$ for all $j=1,2, \ldots, n$, where $x_1 ^\prime, x_2 ^\prime, \ldots, x_n ^\prime \in L_1(\mathcal{M})_{sa}$ are distinct elements of $x_1, x_2, \ldots, x_n$ and $\mathcal{N}_j = W^*(x_j),\,j= 1, 2, \ldots, n$. Then
\begin{eqnarray*}
{\rm var}(y) \leq \frac{1}{2} \sum_{j=1}^n \tau\left((y - y_j^\prime)^2\right).
\end{eqnarray*}
\end{theorem}
\begin{proof}
First we claim that
\begin{equation*}
\tau \left((y - \mathcal{E}_j(y))^2\right) \leq \frac{1}{2} \tau \left((y - y_j^\prime)^2\right).
\end{equation*}
In light of
\begin{eqnarray*}
\frac{1}{2} \tau \left((y - y_j^\prime)^2\right) &=& \frac{1}{2} \left[ \tau \left(y^2\right) - 2\tau (yy_j^\prime) + \tau \left((y_j^\prime)^2\right) \right]\\
&\geq & \frac{1}{2} \left[ 2\tau \left(y^2\right) - 2\tau (yy_j^\prime) \right] \quad (\text{as $\tau\left(y^2\right) \leq \tau\left( (y_j^\prime)^2 \right)$})\\
&\geq & \tau\left(y^2\right) - \tau \left( \mathcal{E}_j (y)^2 \right) \qquad (\text{as $\tau(yy_j^\prime) \leq \| \mathcal{E}_j(y)\|_2^2$})\\
&=& \tau \left((y - \mathcal{E}_j(y))^2\right),
\end{eqnarray*}
and Lemma \ref{le1}, we get the desired inequality.
\end{proof}

Note that in our noncommutative version of Efron--Stein's inequality, we replace the strong condition ``$y$ and $y_j$ are conditionally to $W^*\left( x_1 ,\ldots, x_{j-1}, x_{j+1}, \ldots, x_n \right)$ independent identically distributed random variables'' by the weaker condition ``$\tau(yy_j^\prime) \leq \| \mathcal{E}_j(y)\|_2^2$ and $\tau \left((y_j^\prime)^2\right) \geq \tau \left( y^2 \right)$''. The next result is a commutative version of Efron--Stein's inequality.
\begin{corollary}\cite[Theorem 3.1]{Bouch}
Suppose that $(\Omega, \mathcal{F}, \mathbb{P})$ is a probability space and $f: \Omega^n \longrightarrow \mathbb{R}$ is a measurable function of $n$ variables. Let $X_1, X_2, \ldots, X_n$ be arbitrary independent random variables and $X_1^\prime, X_2^\prime, \ldots, X_n^\prime$ be independent copies of $X_1, X_2, \ldots, X_n$. If $Z= f\left(X_1, X_2, \ldots, X_n\right)$ and $Z_j^\prime = f\left(X_1, \ldots, X_j^\prime, \ldots, X_n\right)$, then
\begin{equation*}
{\rm var}(Z) \leq \frac{1}{n} \sum_{j=1}^n \mathbb{E}\left[\left(Z - Z_j^\prime \right)^2\right].
\end{equation*}
\end{corollary}
\begin{proof}
It is known that if $X$ and $Y$ are independent and identically distributed random variables, then ${\rm var}(X) = \frac{1}{2} \mathbb{E}\left[ (X - Y)^2 \right]$. Therefore $Z_j^\prime$ is an independent copy of $Z$ conditionally to $(X_1, \ldots,X_{j-1}, X_{j+1}, \ldots, X_n)$; cf. \cite[Theorem 3.1]{Bouch}. Hence $\mathbb{E}^{(j)} |ZZ_j^\prime| = \left( \mathbb{E}^{(j)} (Z)\right)^2 $ and $\mathbb{E} \left[ (Z_j^\prime)^2 \right] = \mathbb{E} \left[ Z^2 \right]$, $\, \, (1 \leq j \leq n)$, where $\mathbb{E}^{(j)}$ is the conditional expectation with respect to $(X_1, \ldots,X_{j-1}, X_{j+1}, \ldots, X_n)$. Hence the result can be concluded from Theorem \ref{Efr1}.
\end{proof}

\section{Noncommutative matrix Efron--Stein inequality}

Employing Theorem \ref{Efr1} we state a noncommutative version of the matrix Efron--Stein inequality in this section.

\begin{theorem}[Noncommutative matrix Efron--Stein inequality]\label{pr1}
Let $\mathcal{U}_j$ be a sequence of either freely independent or tensor independent von Neumann subalgebras of $\left(\mathbb{M}_d (\mathcal{M}), \overline{\tau} \right)$ and $U_j \in L_1(\mathcal{U}_j) \, \, (1 \leq j \leq n)$ be self-adjoint matrices and let $F: L_1\left(\mathbb{M}_d (\mathcal{M})\right)_{sa} \times \ldots \times L_1\left(\mathbb{M}_d (\mathcal{M})\right)_{sa} \longrightarrow L_1\left(\mathbb{M}_d (\mathcal{M})\right)_{sa}$ be an arbitrary function.\\
Let $V = F \left( U_1, \ldots, U_n \right)$ and  $V^{(j)} = F \left( U_1, \ldots, U_{j-1}, U_j^\prime, U_{j+1}, \ldots, U_n \right)$ be such that $\overline{\tau} \left( V V^{(j)} \right) \leq \| \mathcal{E}_j \left( V \right) \|_2^2$ and $\overline{\tau} \left( (V^{(j)})^2 \right) \geq \overline{\tau} \left( V^2 \right)$ for all $j = 1, 2, \ldots, n$, where $ U_1^\prime, U_2^\prime, \ldots, U_n^\prime \in L_1\left(\mathbb{M}_d (\mathcal{M})\right)_{sa}$ and $\mathcal{E}_j$ denotes the conditional expectation of $\mathbb{M}_d (\mathcal{M})$ with respect to $\mathcal{U}_j$. Then the inequality
\begin{equation*}
\tau \left( {\rm tr} \left( \left( V - \tau \left( V \right) \right)^2 \right) \right) \leq \frac{1}{2} \tau \left( {\rm tr} \left( \sum_{j=1}^n \left(V - V^{(j)} \right)^2 \right) \right),
\end{equation*}
holds.
\end{theorem}
\begin{proof}
Utilizing Theorem \ref{Efr1} to the von Neumann algebra $\mathbb{M}_d (\mathcal{M})$, we obtain
\begin{eqnarray*}
\tau \left( \overline{{\rm tr}} \left( \left( V - \tau \left( \overline{{\rm tr}} \left( V \right) \right) \right)^2 \right) \right) \leq
\frac{1}{2} \tau \left( \overline{{\rm tr}} \left( \sum_{j=1}^n \left( V - V^{(j)} \right)^2 \right) \right).
\end{eqnarray*}
Hence
\begin{eqnarray*}
\tau \left( {\rm tr} \left( \left( V - \tau \left( \overline{{\rm tr}} \left( V \right) \right) \right)^2 \right) \right) \leq
\frac{1}{2} \tau \left( {\rm tr} \left( \sum_{j=1}^n \left( V - V^{(j)} \right)^2 \right) \right).
\end{eqnarray*}
It is therefore enough to show that
\begin{equation*}
\tau \left( {\rm tr} \left( \left( V - \tau \left( \overline{{\rm tr}} \left( V \right) \right) \right)^2 \right) \right) \geq
\tau \left( {\rm tr} \left( \left( V - \tau (V) \right) ^2 \right) \right),
\end{equation*}
or, equivalently,
\begin{align}\label{tr}
\tau \left( {\rm tr} \left( \left( \tau \left( \overline{{\rm tr}} (V) \right) \right)^2 \right) \right) - 2\tau \left( {\rm tr} \left( V \tau \left( \overline{{\rm tr}} (V) \right) \right) \right) \geq
\tau \left( {\rm tr} \left( \left( \tau (V) \right)^2 \right) \right) - 2\tau \left( {\rm tr} \left( V \tau (V) \right) \right).
\end{align}
This is in turn equivalent to
\begin{equation}\label{tr2}
\frac{1}{d} \left( {\rm tr} \left( \tau (V) \right) \right)^2 \leq {\rm tr} \left( \left( \tau (V) \right)^2 \right).
\end{equation}
Indeed, the left hand side of inequality \eqref{tr} is
\begin{eqnarray*}
&&\hspace{-2cm}\tau \left( {\rm tr} \left( \left( \tau \left( \overline{{\rm tr}} (V) \right) \right)^2 \right) \right) - 2\tau \left( {\rm tr} \left( V \tau \left( \overline{{\rm tr}} (V) \right) \right) \right)\\
&=& \tau \left( {\rm tr} \left( \frac{1}{d^2} \left( {\rm tr} \left( \tau (V) \right) \right)^2 \right) \right) - \frac{2}{d} \tau \left( {\rm tr} \left( V {\rm tr} \left( \tau (V) \right) \right) \right) \\
&=& \frac{d}{d^2} \tau \left( \left( {\rm tr} \left( \tau (V) \right) \right)^2 \right) - \frac{2}{d} \left( {\rm tr} \left( \tau (V) \right) \right) ^2\\
&=& - \frac{1}{d} \left( {\rm tr} \left( \tau (V) \right) \right) ^2,
\end{eqnarray*}
and the right hand side is
\begin{eqnarray*}
\tau \left( {\rm tr} \left( \left( \tau (V) \right)^2 \right) \right) - 2\tau \left( {\rm tr} \left( V \tau (V) \right) \right) &=& {\rm tr} \left( \left( \tau (V) \right)^2 \right) - 2 {\rm tr} \left( \left( \tau (V) \right)^2 \right)\\
&& \qquad \quad \quad (\text{by linearity of $\tau$})\\
&=& - {\rm tr} \left( \left( \tau (V) \right)^2 \right).
\end{eqnarray*}
Inequality \eqref{tr2} can be deduced from Kadison's inequality \cite{K} for positive unital map $\overline{{\rm tr}}$, as
\begin{eqnarray*}
\left( \overline{{\rm tr}} (\tau (V)) \right)^2 \leq \overline{{\rm tr}} \left( \left( \tau (V) \right) ^2 \right),
\end{eqnarray*}
which is equivalent to
\begin{equation*}
\frac{1}{d} \left( {\rm tr} (\tau (V)) \right)^2 \leq {\rm tr} \left( \left( \tau (V) \right) ^2 \right)
\end{equation*}
as desired.
\end{proof}
Applying Theorem \ref{pr1}, we give a version of the Efron--Stein inequality for random matrices. Assume that $\left( X_1, X_2, \ldots, X_n \right)$ is a random vector of mutually independent random variables on a probability space $(\Omega, \mathcal{F}, \mathbb{P})$.\\

Let $F : \mathbb{R}^n \longrightarrow \mathbb{H}_d$ be a bounded measurable function, where $\mathbb{H}_d$ denotes the space of the $d \times d$ Hermitian matrices in $\mathbb{M}_d \left( \mathbb{C} \right)$. Set the random matrix
\begin{eqnarray*}
\textbf{Z} := \textbf{Z} \left( X_1, X_2, \ldots, X_n \right) := F \left( X_1, X_2, \ldots, X_n \right) - \mathbb{E} F \left( X_1, X_2, \ldots, X_n \right).
\end{eqnarray*}
Now suppose that $X_1^\prime, X_2^\prime, \ldots, X_n^\prime$ are independent copies of $X_1, X_2, \ldots, X_n$. Consider the random vectors
\begin{eqnarray*}
\left( X_1, \ldots, X_j^\prime, \ldots, X_n \right).
\end{eqnarray*}
Construct the random matrices
\begin{eqnarray*}
\textbf{Z}^{(j)} := \textbf{Z}^{(j)} \left( X_1, \ldots, X_j^\prime, \ldots, X_n \right) := F \left( X_1, \ldots, X_j^\prime, \ldots, X_n \right) - \mathbb{E} F \left( X_1, X_2, \ldots, X_n \right).
\end{eqnarray*}
With the above notations, the matrix Efron--Stein inequality reads as follows.
\begin{corollary}[Matrix Efron--Stein inequality]
The inequality
\begin{eqnarray*}
\mathbb{E} \left\| \textbf{Z} \right\|_{2}^2 \leq \frac{1}{2} \mathbb{E} \left\| \sum_{j=1}^n \left( \textbf{Z} - \textbf{Z}^{(j)} \right)^2 \right\|_{1},
\end{eqnarray*}
holds, where $\| . \|_{p}$ denote the Schatten $p$-norm.
\end{corollary}
\begin{proof}
Note that $\textbf{Z}$ and $\textbf{Z}^{(j)}$ have the same distribution; cf \cite[page 7] {Pa}. Now the desired inequality can be deduced from Theorem \ref{pr1} by assuming $\mathcal{M} = \mathcal{L}^{\infty} \left( \Omega, \mathbb{P} \right)$.
\end{proof}
A Steele type inequality in the noncommutative setting can be stated as follows.
\begin{theorem}\label{th3}
Let $\mathcal{N}_j\, \,(1 \leq j \leq n)$ be either freely independent or tensor independent von Neumann subalgebras of $\mathcal{M}$ and $x_j \in L_1(\mathcal{N})_j$ be self-adjoint elements for any $1 \leq j \leq n$. Suppose that $f: L_1(\mathcal{M})_{sa} \times \ldots \times L_1(\mathcal{M})_{sa} \longrightarrow L_1(\mathcal{M})_{sa}$ is a function of $n$ variables and let $y= f\left(x_1, x_2, \ldots, x_n\right)$ and $y_j = f_j \left(x_1, \ldots, x_{j-1}, x_{j+1}, \ldots, x_n\right)$ for $j=1,2, \ldots, n$  be self-adjoint elements for some $(n-1)$-variables function $f_j : L_1(\mathcal{N}_1) \times \ldots \times L_1(\mathcal{N}_{j-1}) \times L_1(\mathcal{N}_{j+1}) \times \ldots \times L_1(\mathcal{N}_n) \longrightarrow L_1\left(W^*\left(\bigcup_{k \neq j} \mathcal{N}_k\right)\right)$. Then
\begin{eqnarray*}
{\rm var}(y) \leq  \sum_{j=1}^n \tau\left((y - y_j)^2\right).
\end{eqnarray*}
\end{theorem}
\begin{proof}
To achieve the result, first we show that
\begin{eqnarray}\label{lem2}
\tau\left((y - \mathcal{E}_j(y))^2\right) \leq \tau \left((y - y_j)^2\right).
\end{eqnarray}
For the left hand side term, we have
\begin{eqnarray*}
\tau\left((y - \mathcal{E}_j(y))^2\right) &=& \tau\left(y^2\right) -2 \tau(y \mathcal{E}_j(y)) + \tau\left(\mathcal{E}_j(y)^2\right)\\
&=& \tau\left(y^2\right) -2 \tau\left(\mathcal{E}_j ( y \mathcal{E}_j(y))\right) + \tau\left(\mathcal{E}_j(y)^2\right)\\
&& \qquad\qquad (\text{by (iii) of properties of the conditional expectation})\\
&=& \tau\left(y^2\right) -2 \tau\left(\mathcal{E}_j(y)^2\right) + \tau\left(\mathcal{E}_j(y)^2\right)\\
&& \qquad\qquad (\text{by (ii) of properties of the conditional expectation})\\
&=& \tau\left(y^2\right) - \tau\left(\mathcal{E}_j(y)^2\right),
\end{eqnarray*}
and for the right hand side term, we have
\begin{eqnarray*}
\tau \left((y - y_j)^2\right) &=& \tau\left(y^2\right) - 2 \tau(yy_j) + \tau\left(y_j^2\right)\\
&=& \tau\left(y^2\right) - 2 \tau\left(\mathcal{E}_j(yy_j)\right) + \tau\left(y_j^2\right)\\
&=& \tau\left(y^2\right) - 2 \tau\left(\mathcal{E}_j(y)\mathcal{E}_j(y_j)\right) + \tau\left(y_j^2\right)\\
&=& \tau\left(y^2\right) - 2 \tau\left(\mathcal{E}_j(y)\mathcal{E}_j(y_j)\right) + \tau\left(\mathcal{E}_j(y_j)^2\right).
\end{eqnarray*}
Now inequality \eqref{lem2} follows from
\begin{eqnarray*}
\tau \left((y - y_j)^2\right)  - \tau\left((y - \mathcal{E}_j(y))^2\right) &=&
 \tau\left(y^2\right) - 2 \tau\left(\mathcal{E}_j(y)\mathcal{E}_j(y_j)\right) + \tau\left(\mathcal{E}_j(y_j)^2\right)\\
&& \quad  - \left( \tau\left(y^2\right) - \tau\left(\mathcal{E}_j(y)^2\right) \right)\\
&=& \tau \left( \mathcal{E}_j(y_j)^2 - 2 \mathcal{E}_j(y_j)\mathcal{E}_j(y) + \mathcal{E}_j(y)^2 \right)\\
&=& \tau\left((\mathcal{E}_j(y_j) - \mathcal{E}_j(y))^2 \right)\\
&\geq & 0.
\end{eqnarray*}
It follows from Lemma \ref{le1} and inequality \eqref{lem2} that
\begin{eqnarray*}
{\rm var}(y) \leq  \sum_{j=1}^n \tau\left((y - \mathcal{E}_j(y))^2\right)\leq  \sum_{j=1}^n \tau \left(\left(y - y_j)\right)^2\right).
\end{eqnarray*}
\end{proof}
The following result gives a norm inequality being interesting on its own right.
\begin{proposition}
Let $x_j$ $\, \, (1 \leq j \leq n)$ be either freely independent or tensor independent self-adjoint elements . Then
\begin{eqnarray*}
\left\| \sum_{j=1}^n x_j \right\|_2^2 \leq \left\| \sum_{j=1}^n x_j \right\|_1^2 + \sum_{j=1}^n \| x_j \|_2^2.
\end{eqnarray*}
\end{proposition}
\begin{proof}
Put $y := S_n = \sum_{j=1}^n x_j$.
Lemma \ref{le1} and inequality \eqref{lem2} in the proof of Theorem \ref{th3} with $y_j = y - x_j$ yield that
\begin{eqnarray*}
{\rm var}(y) \leq \sum_{j=1}^n \tau\left((y - \mathcal{E}_j(y))^2\right)
\leq \sum_{j=1}^n \tau \left(\left(y - (y- x_j)\right)^2\right)
= \sum_{j=1}^n \tau \left(x_j^2\right)\,.
\end{eqnarray*}
\end{proof}

\begin{remark}
Another perspective of research concerns discussing the results for the $q$-Guassian random variables; cf. \cite{BEH}. It is interesting to provide a version of Efron--Stein inequality for the $q$-Guassian random variables coming from Coxeter groups of type B. We leave this as a problem for the interested readers.
\end{remark}

\end{document}